\documentclass[12pt]{amsart}

\usepackage[all]{xy}

\newtheorem{proposition}{Proposition}[section]
\newtheorem{theorem}[proposition]{Theorem}
\newtheorem{lemma}[proposition]{Lemma}
\newtheorem{corollary}[proposition]{Corollary}

\newtheorem{problem}[proposition]{Problem}

\theoremstyle{definition}
\newtheorem{definition}[proposition]{Definition}

\theoremstyle{remark}
\newtheorem{remark}[proposition]{Remark}

\makeatletter
\renewcommand\l@subsection{\@tocline{2}{0pt}{2pc}{5pc}{}}
\renewcommand\l@subsubsection{\@tocline{3}{0pt}{4pc}{10pc}{}}
\makeatother
\def\C{\mathbb{C}}
\def\HH{\mathbb{H}}

\def\ov{\overline}

\def\rank{{\rm rank}\,}
\def\sdet{{\rm Sdet}}

\def\Sp{{\rm Sp}}
\def\su{{\rm SU}}

\def\uu{{\rm U}}

\newcommand{\cat}{\mathop{\mathrm{cat}}}
\newcommand{\diag}{\mathrm{diag}}		
\newcommand{\Hprod}[1]{\langle #1\rangle}

\newcommand{\block}[2]{\addtolength{\arraycolsep}{-3pt}\begin{bmatrix}#1\cr#2\end{bmatrix}}

\title[Relative SVD]{Relative singular value decomposition 
\\and applications to LS-category}
\thanks{The three authors are  partially supported by the MINECO and 
FEDER research project MTM2016-78647-P}

\author{E. Mac\'{\i}as-Virg\'os}
\address[E. Mac\'{\i}as-Virg\'os]{Institute of Mathematics, University of Santiago de Compostela, 15782 Spain.}
 \email{quique.macias@usc.es}

\author{M.J. Pereira-S\'aez}
\address[M.J. Pereira-S\'aez]{Facultade de Econom\'{\i}a e Empresa, Universidade da Coru\~na, 15071 Spain.}
\email{maria.jose.pereira@udc.es}

\author{Daniel Tanr\'e}
\address[Daniel Tanr\'e]{D\'epartement de Math\'ematiques, Facult\'e des Sciences et Technologies, Universit\'e de Lille, 59655 Villeneuve d'Ascq Cedex, France.}
\email{Daniel.Tanre@univ-lille.fr}

\begin{document}

\allowdisplaybreaks

\begin{abstract}
Let $\Sp(n)$ be the symplectic group of  quaternionic $(n\times n)$-matrices. 
For any $1\leq k\leq n$, an element $A$ of $\Sp(n)$
can be decomposed in
$A=
\begin{bmatrix}
\alpha&T\cr
\beta&P
\end{bmatrix}$
with $P$ a $(k\times k)$-matrix. 
In this work, starting from a singular value decomposition of $P$, we obtain
what we call a relative singular value decomposition of $A$.
This feature is well adapted for the study of the quaternionic Stiefel manifold $X_{n,k}$, and we 
apply it to the determination of the 
Lusternik-Schnirelmann category of $\Sp(k)$ in $X_{2k-j,k}$, for $j= 0,\,1,\,2$.\end{abstract}


\maketitle

\section{Introduction}

Let $\HH^n$ be the quaternionic $n$-space (with the structure of a right $\HH$-vector space) endowed with the Hermitian product $\Hprod{u,v}=u^*v$. 
For $0< k\leq n$, we denote by $\Sp(n)$ the Lie group of matrices, $A\in\HH^{n\times n}$, such that
$AA^*=I_{n}$ and by $X_{n,k}$ 
the Stiefel manifold of linear maps $\phi\colon \HH^k \to \HH^n$ which preserve the  Hermitian product. Alternatively, the elements of $X_{n,k}$ are the orthonormal $k$-frames of $\HH^n$, represented by a matrix $x\in\HH^{n\times k}$ such that $x^*x=I_k$. Usually we shall write $x=\begin{bmatrix}T\cr P\cr\end{bmatrix}$, with $P\in\HH^{k\times k}$.
Let $\phi_0\in X_{n,k}$ 
be the inclusion $v\mapsto \block{0}{v}$,  represented by the matrix $x_0=\begin{bmatrix}0\cr I_k\cr\end{bmatrix}$.

The linear left action of $\Sp(n)$ on $X_{n,k}$ is transitive and the isotropy group of 
$x_{0}$
is isomorphic to $\Sp(n-k)$.
Therefore the Stiefel manifold $X_{n,k}$ is diffeomorphic to $\Sp(n)/\Sp(n-k)$ and there is a principal fibration
$$\Sp(n-k)\xrightarrow[]{\iota} \Sp(n)\xrightarrow[]{\rho} X_{n,k}.$$

If we write  
$A=
\begin{bmatrix}
\alpha&T\cr
\beta&P
\end{bmatrix}\in\Sp(n)$,
with $T\in\HH^{(n-k)\times k}$ and
$P\in \HH^{k\times k}$, the application $\rho\colon \Sp(n)\to X_{n,k}$ is defined by
$\rho(A)=\block{T}{P}$.
If $P\in\Sp(k)$, we may choose $T=0$ and get an element of $X_{n,k}$. This gives a canonical inclusion,
$$\iota_{n,k}\colon \Sp(k)\to X_{n,k}.$$

\noindent We  shall come back below on some aspects of this inclusion.
First, 
we characterize 
the matrices $P\in\HH^{k\times k}$  that can be completed 
with $T\in\HH^{ (n-k)\times k}$ for getting an element $\block{T}{P}\in X_{n,k}$.
In Proposition~\ref{prop:etrestiefelornot}, we  prove that such  $T$  exists
if, and only if,
the eigenvalues of $P^*P$ (that is, the singular values of $P$), belong to the interval $[0,1]$ and the multiplicity of the eigenvalue~1 is greater  than  or equal to 
 $2k-n$.

\medskip

Next, we use the well-known  (\cite{MR2990115}) singular value decomposition (SVD, in short) of 
 $P\in \HH^{k\times k}$ for the determination of the possible completions of it in an element of $X_{n,k}$. 
More precisely, in Theorem~\ref{thm:shapespn} starting from the SVD of 
$P\in\HH^{k\times k}$, satisfying the previous criterion, we describe the various matrices
of $\Sp(n)$ of the shape
$\begin{bmatrix}
\alpha&T\cr
\beta&P\cr
\end{bmatrix}$. This gives a ``relative SVD of a matrix in $\Sp(n)$''.

\medskip

We apply this  decomposition to the study of the Lus\-ter\-nik-Sch\-ni\-rel\-mann category (in short LS-category).
Let us recall first that an open subset $U$ of a topological space $X$ is called categorical 
  if $U$ is contractible in $X$.
The LS-category, $\cat X$, of  $X$ is defined as the least integer $m\geq 0$ such that $X$ admits a covering by $m+1$ categorical open sets (\cite{CornLuptOprTan2003}). 

The LS-category is a homotopy invariant
that turns out  to be useful in areas such as dynamical systems and symplectic geometry. But it is
also particularly difficult to compute. A longstanding problem is the determination of the LS-category of 
 Lie groups.
In the case of unitary and special unitary Lie groups,  
Singhof determined $\cat \uu(n)=n$ and $\cat\su(n)=n-1$ (\cite{Singhof75}), using eigenvalues. 
This method cannot be carried out for the symplectic groups $\Sp(n)$ 
due to the non-commutativity of  quaternions (\cite{MR2566482}). Some
progress has been made for small $n$ with
$\cat \Sp(2)=3$ (\cite{MR0182969}), $\cat \Sp(3)=5$  (\cite{MR2022385}),  
or with bounds as
$\cat \Sp(n)\leq \binom{n+1}{2}$ (\cite{MacPer2013}) and 
$\cat \Sp(n)\geq n+2$ when $n\geq 3$ (\cite{MR2039767}). In Proposition~\ref{prop:sp2},
we show how Theorem~\ref{thm:shapespn} supplies an explicit minimal categorical open cover of $\Sp(2)$.

Some partial results also exist for the LS-category of symplectic Stiefel manifolds.
 For instance, in \cite{MR2319267}, Nishimoto proves 
 $\cat X_{n,k}=k$ when $n\geq 2k$,  making use of eigenvalues of associated complex matrices. 
 Different techniques of proof have been given for this result, as the
 use of the Cayley transform in \cite{supercayley}, or Morse-Bott functions in \cite{KadzMimu2011}. 
 Let us also mention that Morse-Bott functions are also present in \cite{MR3621032}, \cite{MR3653866} 
 for the study of LS-category.
 Finally recall the existence of a lower bound for the LS-category of Stiefel manifolds,
 generally better than the classical cup-length,
 established by Kishimoto in \cite{MR2310478}, and recalled in Theorem~\ref{thm:kishi}.

\medskip

In this work, we study  the subspace LS-category of $\Sp(k)$ in $X_{n,k}$, denoted
$\cat_{X_{n,k}}\Sp(k)$. This means that we are
looking for  families of open sets in $X_{n,k}$ covering $\Sp(k)$ whose elements are contractible in $X_{n,k}$. We prove in  Propositions \ref{prop:spnrelative0},
\ref{prop:spnrelative1} and \ref{prop:spnrelative2}
that
$$\cat\nolimits_{X_{2k-j,k}}\Sp(k) \leq \cat \Sp(j), \quad \text{for\ } j=0,\,1,\,2,$$
 \noindent and we wonder if this is still true for any $j\geq 0$.
 
\medskip

{\bf Notations and Conventions.}
For any pair of square matrices (not necessarily of the same size) the relation $A\sim B$ means: ``$A$ is invertible if and only if $B$ is so." 

If $(t_{1},\dots,t_{q})$ is a sequence of quaternions, we denote by
$\diag(t_{i})_{q\times q}$ the $(q\times q)$-matrix having the $t_{i}$'s on the diagonal and 0 otherwise.


\section{Stiefel manifolds}
\begin{quote}
In this section, we consider a matrix $P \in\HH^{k\times k}$  and study the existence of a ``companion'' $T\in\HH^{(n-k)\times k}$ which gives an element 
$\block{T}{P}$ of $X_{n,k}$.
\end{quote}

\medskip

An element   of $X_{n,k}$ can be represented by a matrix
$x=\block{T}{P}$, with $T\in\HH^{(n-k)\times k}$ and
$P\in \HH^{k\times k}$. 
The preservation of the Hermitian product corresponds to the equation
$x^*x=I_{k}$, 
which becomes
$$T^*T+P^*P=I_{k}.$$

\begin{definition}\label{def:admissible}
A   matrix $P\in\HH^{k\times k}$ is 
{\em $n$-admissible}
 if there exists $T\in \HH^{(n-k)\times k}$
such that $\block{T}{P}\in X_{n,k}$.
The integer number $e=2k-n$ is called the \emph{excess} of $X_{n,k}$.
\end{definition}

Admissible matrices can be entirely characterized by eigenvalues.

\begin{proposition}\label{prop:etrestiefelornot}
A matrix $P\in\HH^{k\times k}$ is $n$-admissible
if, and only if,
the eigenvalues of $P^*P$ belong to the interval $[0,1]$ and the multiplicity of the eigenvalue~1 is greater than or equal to 
the excess $e=2k-n$.
\end{proposition}

 Let us notice that  the second condition is automatically verified if $e\leq 0$.

\begin{proof}
Let
$$P=U\,
\begin{bmatrix}I_{p\times p}&0&0\\
0&\diag(t_i)_{q\times q}&0\\
0&0&0_{r\times r}
\end{bmatrix}\,V^*
$$ 
be the SVD  of $P$, with $p+q+r=k$, $U,\,V\in\Sp(k)$, $p,q,r\geq 0$ and  $0<t_i<1$.

$\bullet$ If there exists $T\in\HH^{(n-k)\times k}$ such that
$\block{T}{P}\in X_{n,k}$,
the equality
$T^*T+P^*P=I_k$
implies
$$T^*T=V\,
\begin{bmatrix}
0_{p\times p}&0& 0\\
0&\diag(1-t_i^2)_{q\times q}&0\\
0&0&I_{r\times r}
\end{bmatrix}\,V^*\in \HH^{k\times k}.$$
 As $T^*T$ is hermitian semi-definite positive, we deduce 
$1-t_i^2> 0$ and
$0< t_i< 1$.
For any non-square matrix $T\in \HH^{(n-k)\times k}$, it is known that $\rank (T^*T)=\rank(T)$,
see Lemma~\ref{lem:todo}. This implies
$q+r\leq \min(n-k,k)$
and
$$p=k-(q+r)\geq k-\min(n-k,k)\geq 2k-n=e.$$

$\bullet$ Suppose now $t_i\in ]0,1[$ and $p\geq e$. We consider the matrix

$$T=\,
\begin{bmatrix}
0_{p'\times p}&0&0\\
0&\diag(s_i)_{q\times q}&0\\
0&0&&I_{r\times r}
\end{bmatrix}\,V^*,
$$
with $0<s_i=\sqrt{1-t_i^2}<1$ and $p'+q+r=n-k$. Then we have
$T^*T+ P^*P=I_k$ and 
$\block{T}{P}\in X_{n,k}$.
\end{proof}

Let us recall \emph{the Study determinant} (\cite{ASLAKSEN}) useful for the detection of inversible matrices.
As any quaternionic matrix $M\in\HH^{n\times n}$ can be written as $M=X+\mathbf{j}Y$ with 
$X,\,Y\in \C^{n\times n}$, we  associate to $M$ a complex matrix, $\chi(M)$, defined by
$$\chi(M)=\begin{bmatrix}
X&-\ov{Y}\,\,\cr
Y&\ov{X}
\end{bmatrix}
\in \C^{2n\times 2n}.$$
The Study determinant of $M$, defined by
$\sdet(M)=\sqrt{\det \chi(M)}$,  verifies the following properties.
\begin{enumerate}
\item The matrix $M$ is invertible if, and only if, $\sdet (M)\neq 0$.
\item If $M,\, N\in \HH^{n\times n}$, then $\sdet(MN)=\sdet(M)\;\sdet(N)$.
\item If $N$ is obtained from $M$ by adding  a left multiple of a row to another row
 or a right multiple of a column to another column, then we have
$\sdet(M)=\sdet(N)$.
\item If $M$ is a triangular matrix then $\sdet(M)$ equals $\vert m_{11}\cdots m_{nn}\vert$, the norm of the product of the elements of the diagonal.
\end{enumerate}

We complete these properties by the  following one, 
well adapted to the quaternionic matrices appearing in the last sections.

\begin{lemma}\label{lem:sylvester}
Let $M\in\HH^{m\times n}$ and $N\in\HH^{n\times m}$. Then we have
$$\sdet(I_m+MN)=
\sdet(I_n+NM).
$$
\end{lemma}

\begin{proof}
This is a classical argument,
\begin{align*}
&\sdet \begin{bmatrix}
I_m+MN&-M\cr
0&I_n
\end{bmatrix}\\
=&
\sdet
\begin{bmatrix}
I_m&-M\cr
N&I_n
\end{bmatrix}
=\sdet
\begin{bmatrix}
I_{m}&0\cr
N&I_n+NM\cr
\end{bmatrix}.\qedhere
\end{align*}
\end{proof}

We end this section with the following lemma, used in the proof of Theorem~\ref{thm:shapespn}. 
It is a classical result and we give the proof for the 
convenience of the reader.

\begin{lemma}\label{lem:todo}
Let $M\in\HH^{m\times n}$ be a non-necessarily square quaternionic matrix. Then, we have
$\ker M^*M=\ker M$ and $\ker MM^*=\ker M^*$.
\end{lemma}

\begin{proof}
The inclusion $\ker M\subset \ker M^*M$ is direct. On the other hand, if $u\in\ker M^*M$,  we get
$|M(u)|^2=\langle Mu,Mu\rangle=\langle u,M^*Mu\rangle=0$ and $u\in \ker M$.
A similar argument gives the second equality.
\end{proof}


\section{Relative singular value decomposition  in  $\Sp(n)$}
\begin{quote}
In this section, we establish a ``relative singular value decomposition''  of  the elements of $\Sp(n)$. This structure
proves to be effective for the study of the injection $\Sp(k)\to X_{n,k}$ as it appears in Section~\ref{sec:levelone}.
\end{quote}

\medskip

\begin{theorem}\label{thm:shapespn}
For any $k\leq n$,  an element $A$ of $\Sp(n)$ can be written in blocks as follows,
$$A=\addtolength{\arraycolsep}{-2pt}\begin{bmatrix}
m
\begin{bmatrix}
I_{p'}&0&0\cr
0&\diag(\cos \theta_i)_{q\times q}&0\cr
0&0&0_r
\end{bmatrix}
\ell^*
&
m
\begin{bmatrix}
0_{p'\times p}&0&0\cr
0&-\diag(\sin \theta_i)_{q\times q}&0\cr
0&0&-I_r
\end{bmatrix}
b^*\cr 
&\cr
a\begin{bmatrix}
0_{p\times p'}&0&0\cr
0&\diag(\sin \theta_i)_{q\times q}&0\cr
0&0&I_r
\end{bmatrix}
\ell^*
&
a\begin{bmatrix}
I_{p}&0&0\cr
0&\diag(\cos \theta_i)_{q\times q}&0\cr
0&0&0_r
\end{bmatrix}
b^*
\end{bmatrix},
$$
with $\theta_i\in ]0,\pi/2[$, $a, b\in\Sp(k)$, $m,\ell\in\Sp(n-k)$, 
$p\geq 2k-n$,
$p+q+r=k$ and $p'+q+r=n-k$.
\end{theorem}

\begin{proof} 
Let $A=
\begin{bmatrix}
\alpha&T\cr
\beta&P
\end{bmatrix}\in\Sp(n)$,
with $P\in\HH^{k\times k}$. The SVD  of $P$ gives
$$P=
a
\begin{bmatrix}
I_{p}&0&0\cr
0&\diag( c_i  )_{q\times q}&0\cr
0&0&0_r
\end{bmatrix}
b^*,
$$
with $a,\,b\in\Sp(k)$, $p+q+r=k$ and $0<c_i<1$. From $T^*T+P^*P=I_{k}$, we deduce
\begin{eqnarray}
T^*T&=&I_{k}-
b\begin{bmatrix}
I_{p}&0&0\cr
0&\diag(c_i^2)_{q\times q}&0\cr
0&0&0_r
\end{bmatrix}
b^*\nonumber \\
&=&
b\begin{bmatrix}
0_{p}&0&0\cr
0&\diag(s^2_{i})_{q\times q}&0\cr
0&0&I_r
\end{bmatrix}
b^*,\label{equa:T*T}
\end{eqnarray}
with $0<s_{i}<1$ and $s_{i}^2=1-c_i^2$.
We proceed in three steps, determining successively $T$, $\beta$ and $\alpha$.

\medskip

\emph{Step 1.} Let $p'$ such that $p'+q+r=n-k$. The matrix $T$ can be written as
\begin{equation}\label{equa:T}
T=m
\begin{bmatrix}
0_{p'\times p}&0&0\cr
0&-\diag(s_i)_{q\times q}&0\cr
0&0&-I_r
\end{bmatrix}
b^*
\end{equation}
if, and only if,  the columns $(b_{i})_{1\leq i\leq k}$
and $(m_{j})_{1\leq j\leq n-k}$ of the matrices $b$ and $m$, respectively, verify:
\begin{eqnarray}
T(b_i)=0
&\text{ for }&
1\leq i\leq p, \label{equa:T1}\\
T(b_{p+i})=-m_{p'+i}s_{i}
&\text{ for }&
1\leq i\leq q, \label{equa:T2}\\
T(b_{p+q+i})=m_{p'+q+i}
&\text{ for }&
1\leq i\leq r.   \label{equa:T3}
\end{eqnarray}
We therefore have to establish these three properties from (\ref{equa:T*T}).
The equation (\ref{equa:T1}) is a direct consequence of Lemma~\ref{lem:todo}.
Next, the equations (\ref{equa:T2}) and (\ref{equa:T3}) define the vectors
$(m_{p'+i})_{1\leq i\leq q+r}$. They constitute an orthogonal system because the same holds for the corresponding $(b_i)$. 
 In fact, from (\ref{equa:T*T}) and (\ref{equa:T2}), we deduce for $1\leq i\leq q$,
$$\langle m_{p'+i},m_{p'+i}\rangle 
=
\frac{1}{s_{i}^2}\langle T(b_{p+i}), T(b_{p+i}) \rangle
=
\frac{1}{s_{i}^2}\langle b_{p+i}, T^*T(b_{p+i}) \rangle
=1,
$$
and analogousy $\langle m_{p'+i},m_{p'+j}\rangle=0$ for $i\neq j$.

We also have
$\langle m_{p'+q+i},m_{p'+q+j}\rangle= \delta_{i,j}$
for $1\leq i,j\leq r$.
Thus it suffices to complete $(m_{p+i})_{1\leq i\leq q+r}$ in an orthonormal basis to get  the  announced expression of $T$.

\medskip

\emph{Step 2.} We determine $\beta\in\HH^{k\times (n-k)}$ such that $\beta\beta^*+PP^*=I_{k}$. This equality gives
\begin{eqnarray}
\beta\beta^*
&=&
I_{k}-
a
\begin{bmatrix}
I_{p}&0&0\cr
0&\diag(c_i^2)_{q\times q}&0\cr
0&0&0_r\cr
\end{bmatrix}
a^*\nonumber \\
&=&
a\begin{bmatrix}
0_{p}&0&0\cr
0&\diag(s^2_{i})_{q\times q}&0\cr
0&0&I_r
\end{bmatrix}
a^*.\label{equa:betabeta*}
\end{eqnarray}
The argumentation developed in the first step brings a matrix $\ell\in\Sp(n-k)$ such that
\begin{equation}\label{equa:beta*}
\beta^*=
\ell
\begin{bmatrix}
0_{p'\times p}&0&0\cr
0&\diag(s_i)_{q\times q}&0\cr
0&0&I_r
\end{bmatrix}
a^*.
\end{equation}
As in the first step, the columns $(\ell_{p'+i})_{1\leq i\leq q+r}$ are explicitly determined, but for the family
$(\ell_{i})_{1\leq i\leq p'}$ the only requirement is to have an orthonormal basis
$(\ell_{j})_{1\leq j\leq p'+q+r}$.

\medskip
\emph{Step 3.} We are reduced to decompose the matrix
$\alpha\in\HH^{(n-k)\times (n-k)}$. From $AA^*= I_n=A^*A$, we deduce
\begin{align}
\alpha\alpha^*+TT^*
=\ &
I_{n-k}, \label{equa:alphaalpha*}\\
\alpha^*T+\beta^*P
=\ &
0,   \label{equa:alphaalpha}\\
\alpha^*\alpha+\beta^*\beta
=\ &
I_{n-k}. \label{equa:alpha*alpha}
\end{align}
As before, we denote by $(m_{j})_{1\leq j\leq n-k}$ and $(\ell_{j})_{1\leq j\leq n-k}$ the columns of the matrices 
$m$ and $\ell$ respectively.
Replacing $T$ by its value (\ref{equa:T}), we deduce from (\ref{equa:alphaalpha*}) that the family 
$(m_{p'+q+i})_{1\leq i\leq r}$ is a basis of $\ker\alpha\alpha^*=\ker\alpha^*$, see Lemma~\ref{lem:todo}.

\smallskip

The replacement of $T$, $\beta$, $P$ by their value in (\ref{equa:alphaalpha}) gives the equality
$$\alpha^*m
\begin{bmatrix}
0_{p'\times p}&0&0\cr
0&-\diag(s_i)_{q\times q}&0\cr
0&0&-I_r
\end{bmatrix}
=-\ell
\begin{bmatrix}
0_{p'\times p}&0&0\cr
0&\diag(s_ic_{i})_{q\times q}&0\cr
0&0&0_r
\end{bmatrix},
$$

which implies the relations
\begin{align}
\alpha^*m_{p^\prime+i}=\ell_{p^\prime + i}\,c_i,\quad &\text{for\ } 1\leq i\leq q,\\
\alpha^*m_{p^\prime+q+i}=0, \quad &\text{for\ } 1\leq i\leq r.
\end{align}
Thus, for proving
\begin{equation}\label{equa:alpha}
\alpha=m
\begin{bmatrix}
I_{p'}&0&0\cr
0&\diag(c_{i})_{q\times q}&0\cr
0&0&0_r
\end{bmatrix}
\ell^*,
\end{equation}
it remains to establish  
\begin{equation}\label{equa:ultima}
\alpha \ell_{i}=m_{i},  \quad \text{for\ } 1\leq i\leq p'.
\end{equation}
For that, starting from an orthonormal basis $(\ell_{j})_{1\leq j\leq n-k}$ built in  Step~2, we have to prove that we  could have taken
$m_{i}=\alpha \ell_{i}$ for $1\leq i\leq p'$  in order to complete an orthonormal basis $(m_{j})_{1\leq j\leq n-k}$ as  we built in Step 1.

\smallskip

From (\ref{equa:alpha*alpha}) and (\ref{equa:beta*}), we deduce
$$\alpha^*\alpha=\ell
\begin{bmatrix}
I_{p'}&0&0\cr
0&\diag(c^2_{i})_{q\times q}&0\cr
0&0&0_r
\end{bmatrix}
\ell^*,$$
which implies $\alpha^*\alpha \ell_{i}=\ell_{i}$,
for all $1\leq i\leq p'$.
We prove now the orthonormality of $(m_{j})_{1\leq j\leq n-k}$. 

\smallskip

$\bullet$ Let $1\leq j,k\leq  p'$. We have
$$\langle m_{j},m_{k}\rangle=
\langle \alpha \ell_{j},\alpha \ell_{k}\rangle=
\langle \ell_{j},\alpha^*\alpha \ell_{k}\rangle =
\langle \ell_{j},\ell_{k}\rangle,$$
which gives an orthogonality relation for $j\neq k$ and $\langle m_{j},m_{j}\rangle=1$.

\smallskip

$\bullet$ Let $1\leq j \leq p'$ and $1\leq k\leq q$. We have:
\begin{eqnarray*}
\langle m_{j}, m_{p'+k} \rangle
&=&
\langle \alpha \ell_{j}, m_{p'+q+k} \rangle
=
\langle \ell_{j}, \alpha^*m_{p'+q+k} \rangle
\\
&=&
\langle \ell_{j}, \ell_{p'+k} \rangle c_{k}
=0.
\end{eqnarray*}

\smallskip
$\bullet$ Let $1\leq j\leq p'$ and $1\leq k\leq r$. We have:
\begin{eqnarray*}
\langle m_{j}, m_{p'+q+k} \rangle
&=&
\langle \alpha \ell_{j}, m_{p'+q+k}  \rangle
=
\langle \ell_{j}, \alpha^* m_{p'+q+k}  \rangle
\\
&=&
\langle \ell_{j},  0 \rangle=0. \qedhere
\end{eqnarray*}
\end{proof}

The following particular case of Theorem~\ref{thm:shapespn} corresponds to  $k=1$ and $2k-n\leq 0$.

\begin{corollary}\label{cor:shapespn}
Let $n\geq 2$. Any element of $\Sp(n)$ can be written as 
$$P=\begin{bmatrix}
m
\begin{bmatrix}
I_{n-2}&0\cr
0&\cos \theta
\end{bmatrix}
\ell^*
&
m\block{0}{ -\sin\theta} E\cr
&\cr
[0\; \sin\theta]\ell^*&(\cos\theta) E
\end{bmatrix},
$$
with $m,\,\ell\in \Sp(n-1)$, $\cos\theta\in [0,1]$, $E\in \Sp(1)$.
\end{corollary}


\section{Background on LS-category}\label{sec:backLS}
\begin{quote}
We recall  basic definitions and properties of the Lus\-ter\-nik-Sch\-ni\-rel\-mann category 
(LS-category in short). We also
state the results on the LS-category of Stiefel manifolds obtained by T.~Nishimoto
(\cite{MR2319267}) and D.~Kishimoto (\cite{MR2310478}) as well as the technique for the construction of
categorical open subsets, introduced by the authors in \cite{supercayley}.
\end{quote}

\medskip

The definition of LS-category has been recalled in the introduction, see
\cite{CornLuptOprTan2003} for more details. If  $X$ is an $(m-1)$-connected $CW$-complex, then 
there is the upper bound,
$$\cat X\leq (\dim X)/m.$$
As 
$\dim X_{n,k}=
\dim \Sp(n)-\dim\Sp(n-k)
=
k(4n-2k+1)$, we get (see \cite[Proposition 2.1]{MR0431239} for the connectivity of $X_{n,k}$)
\begin{equation}\label{equa:dimconnec}
\cat X_{n,k}
\leq
 \frac{k(4n-2k+1)}{4(n-k)+3}.
\end{equation}
A  lower bound is given by the cup length in the cohomology algebra but, 
for Stiefel manifolds, there is also a lower bound, due to Kishimoto.

\begin{theorem}[\cite{MR2310478}]\label{thm:kishi}
We have
$$\cat X_{n,k}\geq
\left\{
\begin{array}{ccl}
k
&\text{if}&
n\geq 2k-1,\\
k+1
&\text{if}&
n=2k-2\;\text{or}\; n=2k-3,\\
k+2
&\text{if}&
n\leq 2k-4.
\end{array}\right.$$
\end{theorem}
In the particular case $n\geq 2k$, Nishimoto has computed the LS-category of $X_{n,k}$,
using  the number of eigenvalues of an associated complex matrix.

\begin{theorem}[\cite{MR2319267}]\label{thm:nishi}
If $n\geq 2k$ then $\cat X_{n,k}=k$.
\end{theorem}

\begin{remark}
From Theorem~\ref{thm:kishi},  Theorem~\ref{thm:nishi} and (\ref{equa:dimconnec}), we can deduce for instance:
$\cat X_{3,2}=2$ and $ \cat X_{4,3}= 4$.
\end{remark}

Nishimoto's result can also be proven (\cite{supercayley}) from  Cayley open subsets, defined as follows.

\begin{theorem}[{\cite[Theorem 1.2]{supercayley}}]\label{thm:cayleystiefel}
Let $P\in \HH^{k\times k}$ be an $n$-admissible matrix.
The Cayley open subset
$$\Omega(P)=\left\{\block{\tau}{\pi}\in X_{n,k}\mid \pi+P^*\text{ invertible }\right\}$$
is categorical in $X_{n,k}$.
\end{theorem}

\begin{remark}
Let   $\diag({0_s,-I_t,I_r})=\begin{bmatrix}
0_{s}&0&0\cr
0&-I_t&0\cr
0&0&I_{r}
\end{bmatrix}\in \HH^{k\times k}$  be the diagonal matrix defined by blocks from the null matrix
$0_s\in \HH^{s\times s}$ and the identity matrices
$I_t\in\HH^{t\times t}$, $I_r\in\HH^{r\times r}$, with  $s+t+r=k$ and $s,t,r\geq 0$. Then $\diag({0_{s},-I_t,I_r})$ is 
$n$-admissible if and 
only if $r+t\geq e$. In this case, we have the categorical open subset of $X_{n,k}$
$$\Omega({0_s,-I_t,I_r})=\left\{\block{T}{P}\in X_{n,k}\mid P+\diag({0_s,-I_t,I_r})
\text{ invertible }\right\}.$$
\end{remark}

From Theorem~\ref{thm:shapespn}, we determine an explicit minimal  categorical open cover of $\Sp(2)$.

\begin{proposition}\label{prop:sp2}
The four open subsets
$\Omega(I_2)$, $\Omega(-I_2)$, $\Omega(I_1,-I_1)$ and $\Omega(-I_1,I_1)$
constitute a categorical open cover of $\Sp(2)$. 
\end{proposition}

\begin{proof}
From Theorem~\ref{thm:shapespn}, we know that any element of $\Sp(2)$ can be written as
$$P=
\begin{bmatrix}
m\cos\theta \;\ell^*&-m\sin\theta \;b^*\cr
a\sin\theta \;\ell^*&a\cos\theta\; b^*
\end{bmatrix},
$$
where $a$, $b$, $m$, $\ell$ are quaternionic numbers of norm~1 and $\cos\theta\in [0,1]$. We set
$\varepsilon_i=\pm 1$ for $i=1,\,2$ and
$\diag(\varepsilon_1,\varepsilon_2)=
\begin{bmatrix}
\varepsilon_1&0\cr
0&\varepsilon_2
\end{bmatrix}$.
We observe $\diag(\varepsilon_1,\varepsilon_2)^2=I_2$ and we are looking for the property 
``$P+\diag(\varepsilon_1,\varepsilon_2)$ is invertible".
Lemma~\ref{lem:sylvester} and easy calculations imply that

\begingroup
\addtolength{\jot}{0.5em}
\begin{align*}
P+\diag(\varepsilon_1,\varepsilon_2)
\sim\ &
\begin{bmatrix}
m&0\cr
0&a
\end{bmatrix}
\begin{bmatrix}
\cos\theta&-\sin\theta\cr
\sin\theta&\cos\theta
\end{bmatrix}
\begin{bmatrix}
\ell^*&0\cr
0&b^*
\end{bmatrix}
+\diag(\varepsilon_1,\varepsilon_2)\\
\sim\ &
\begin{bmatrix}
\ell^*&0\cr
0&b^*
\end{bmatrix}
\diag(\varepsilon_1,\varepsilon_2)
\begin{bmatrix}
m&0\cr
0&a
\end{bmatrix}
\begin{bmatrix}
\cos\theta&-\sin\theta\cr
\sin\theta&\cos\theta
\end{bmatrix}
+I_2\\
\sim\ &
\begin{bmatrix}
\varepsilon_1 \ell^*m&0\cr
0&\varepsilon_2 b^*a
\end{bmatrix}
\begin{bmatrix}
\cos\theta&-\sin\theta\cr
\sin\theta&\cos\theta
\end{bmatrix}
+I_2\\
\sim\ &
\begin{bmatrix}
\varepsilon_1 \ell^*m&0\cr
0&\varepsilon_2 b^*a
\end{bmatrix}
+
\begin{bmatrix}
\cos\theta&\sin\theta\cr
-\sin\theta&\cos\theta\cr
\end{bmatrix}\\
=\ &
\begin{bmatrix}
\varepsilon_1 \ell^*m+\cos\theta&\sin\theta\cr
-\sin\theta&\varepsilon_2 b^*a+\cos\theta\cr
\end{bmatrix}.
\end{align*}
\endgroup

We set  $Q_1=\varepsilon_1\ell^*m$ and  $Q_2=\varepsilon_2b^*a$.

\medskip

\emph{Suppose $\cos\theta\neq 1$.} This implies $\sin\theta\neq 0$ and we may use it as a ``pivot" in the last matrix. This gives,  by adding to the second row a left multiple of the first row,
$$P+\diag(\varepsilon_1,\varepsilon_2)
\sim
\begingroup
\def\arraystretch{1.4}
\begin{bmatrix}
\cos\theta+ Q_1&\sin\theta\cr
-\sin\theta-(\sin\theta)^{-1}(\cos\theta+ Q_2)(\cos\theta+ Q_1)&0\cr
\end{bmatrix}.
\endgroup
$$
Thus $P+\diag(\varepsilon_1,\varepsilon_2)$ is not inversible if and only if
\begin{eqnarray*}
&& \sin^2\theta+(\cos\theta+Q_2)(\cos\theta+Q_1)=0\\
&\Leftrightarrow&
 1+\cos\theta \,Q_2+(\cos\theta+Q_2)Q_1=0\\
&\Leftrightarrow&
Q_1=- (\cos\theta+Q_2)^{-1}(1+\cos\theta\, Q_2).
\end{eqnarray*}

The last writing makes sense since $\cos\theta\neq 1$ implies  $\cos\theta+Q_2\neq 0$ because $|Q_2|=1$.
If $(\varepsilon_1,\varepsilon_2)$ is given, 
the previous equation admits  a unique solution $(Q_1,Q_2)$.
Therefore, among the matrices of the statement, we can find a matrix $\diag(\varepsilon_1,\varepsilon_2)$
for which $P+\diag(\varepsilon_1,\varepsilon_2)$ is invertible. (In fact, two of them suffice in this case.)

\medskip
\emph{If $\cos\theta=1$}, then we have
$$P+\diag(\varepsilon_1,\varepsilon_2)
\sim
\begin{bmatrix}
 1+ Q^\prime_1&0\\
0&1+ Q^\prime_2\,
\end{bmatrix}.\qedhere$$
\end{proof}
Let us notice that we need the four matrices of the statement to ensure the existence of one case 
such that $P+\diag(\varepsilon_1,\varepsilon_2)$ is invertible. 
In fact, we already know from \cite{MR0182969}
that there is no categorical open cover of $\Sp(2)$ with strictly less than 4 elements.


\section{Subspace LS-category of $\Sp(k)$ \\in the Stiefel manifold $X_{n,k}$}\label{sec:levelone}

\begin{quote}
We give an upper bound for the subspace LS-category,
$\cat_{X_{2k-1,k}}\Sp(k)$,
of $\Sp(k)$ in $X_{n,k}$, for $n\geq 2k$, $n=2k-1$ and $n=2k-2$. A question for the general case is also proposed. 
\end{quote}

\medskip

If $n\geq 2k$, we first notice that the zero matrix $0_k\in\HH^{k\times k}$ is $n$-admissible. Therefore $\Sp(k)$
is included in the categorical open subset $ \Omega(0_k)$ and the next result follows.

\begin{proposition}\label{prop:spnrelative0}
If $0< 2k\leq n$, we have $\cat_{X_{n,k}}\Sp(k) =0$.
\end{proposition}

Consider now the second case.

\begin{proposition}\label{prop:spnrelative1}
If  $n=2k-1$, $0< k $, we have 
$\cat_{X_{2k-1,k}}\Sp(k) \leq 1$.
\end{proposition}

\begin{proof}
Observe that the matrices $\diag(0_{k-1},I_1)$ and $\diag(0_{k-1},-I_1)$ are $(2k-1)$-admissible.
We decompose an element of $P\in\Sp(k)$ as
$$P=\begin{bmatrix}
m
\begin{bmatrix}
I_{k-2}&0\cr
0&\cos \theta
\end{bmatrix}
\ell^*
&
m\block{0}{-\sin\theta} E\cr
&\cr
[0\;\sin\theta]\ell^*&(\cos\theta) E
\end{bmatrix},
$$
with $m,\,\ell\in \Sp(k-1)$, $\cos\theta\in [0,1]$, $E\in \Sp(1)$.

\medskip

$\bullet$ Suppose $1+E\cos \theta\neq 0$. Then we have
\begin{align*}
&P+\diag(0_{k-1},I_1)\\
\sim\ &
m\begin{bmatrix}
I_{k-2}&0\cr
0&\cos\theta
\end{bmatrix} \ell^*
+m\block{0}{\sin\theta}EE^*(E^*+\cos\theta)^{-1}
[0\;\sin\theta]
\ell^*\\
\sim\ &
\begin{bmatrix}
I_{k-2}&0\cr
0&\cos\theta+\sin^2\theta(E^*+\cos\theta)^{-1}
\end{bmatrix}.
\end{align*}

Let us notice that 
$$\cos\theta+\sin^2\theta(E^*+\cos\theta)^{-1}=E^*(E+\cos\theta)(E^*+\cos\theta)^{-1}$$ is a 
quaternion of norm 1. Thus the matrix $P+\diag(0_{k-1},I_1)$ is invertible.

$\bullet$ If $1+E\cos \theta= 0$, then we have $\cos\theta=1$ and $E=-1$. This implies
$P=\begin{bmatrix}
m\ell^*&0\cr
0&-1\cr
\end{bmatrix}$
and $P\in \Omega(0_{k-1},-I_1)$.

In conclusion, we cover $\Sp(k)$ by the two open subsets 
$\Omega(0_{k-1},I_1)$
and
$\Omega(0_{k-1},-I_1)$,
which are contractible in $X_{2k-1,k}$.
\end{proof}

Finally, we state our last result in this direction.

\begin{proposition}\label{prop:spnrelative2}
If  $n=2k-2$, with $1< k $, we have
$\cat_{X_{2k-2,k}}\Sp(k) \leq 3$.
\end{proposition}

\begin{proof}
We prove that the four open subsets 
$\Omega(0_{k-2},1,1)$,
$\Omega(0_{k-2},1,-1)$,
$\Omega(0_{k-2},-1,1)$,
$\Omega(0_{k-2},-1,-1)$
form a categorical open cover of $\Sp(k)$ in $X_{2k-2,k}$.
Let $P\in\Sp(k)$ that we write, by taking,  in Theorem \ref{thm:shapespn}, a block of size $2\times 2$ at the bottom right corner, as
\begin{equation}
P=
\begin{bmatrix}
m
\begin{bmatrix}
I_{k-4}&0&0\cr
0&\cos\theta_1&0\cr
0&0&\cos\theta_2\cr
\end{bmatrix}
\ell^*
&
m
\begin{bmatrix}
0_{k-4,1}&0_{k-4,1}\cr
-\sin\theta_1&0\cr
0&-\sin\theta_2\cr
\end{bmatrix}
b^*
\cr
&\cr
a
\begin{bmatrix}
0_{1,k-4}&\sin\theta_1&0\cr
0_{1,k-4}&0&\sin\theta_2\cr
\end{bmatrix}
\ell^*
&
a
\begin{bmatrix}
\cos\theta_1&0\cr
0&\cos\theta_2\cr
\end{bmatrix}
b^*\cr
\end{bmatrix},
\end{equation}
where $\cos\theta_1,\cos\theta_2\in [0,1]$, $a,b\in\Sp(2)$ and $m,\ell\in\Sp(k-2)$.

\medskip

\emph{First step.} Claim:
if $a^*b\in \Omega
\begin{bmatrix}
\cos\theta_1&0\cr
0&\cos\theta_2\cr
\end{bmatrix}
$
then 
$P\in \Omega(0_{k-2},I_{2})$.

Let 
$H=a\begin{bmatrix}
\cos\theta_1&0\cr
0&\cos\theta_2
\end{bmatrix}
b^*+I_2$.
The hypothesis on $a^*b$ implies the invertibility of $H$. Thus, we can use $H$ as a ``pivot''  to add to the first block of columns the second block multiplied on the right by  
$$X=H^{-1}a 
\begin{bmatrix}
0_{1,k-4}&\sin\theta_1&0\cr
0_{1,k-4}&0&\sin\theta_2\cr
\end{bmatrix}
\ell^*$$ 
and we get
\begingroup
\addtolength{\jot}{0.5em}
\begin{align}
\ &P+\diag(0_{k-2},I_{2})\nonumber\\
\sim\ &
m
\begin{bmatrix}
I_{k-4}&0&0\cr
0&\cos\theta_1&0\cr
0&0&\cos\theta_2
\end{bmatrix}
\ell^* \nonumber
\\
&\quad
- m
\begin{bmatrix}
0_{k-4,1}&0_{k-4,1}\cr
-\sin\theta_1&0\cr
0&-\sin\theta_2
\end{bmatrix}
b^*H^{-1}a
\begin{bmatrix}
0_{1,k-4}&\sin\theta_1&0\cr
0_{1,k-4}&0&\sin\theta_2
\end{bmatrix}
\ell^*\nonumber\\
\sim\ &
\begin{bmatrix}
I_{k-4}&0_{k-4,2}\cr
0_{2,k-4}&
\begin{bmatrix}
\cos\theta_1&0\cr
0&\cos\theta_2
\end{bmatrix}+
\begin{bmatrix}
\sin\theta_1&0\cr
0&\sin\theta_2
\end{bmatrix}
b^*H^{-1}a
\begin{bmatrix}
\sin\theta_1&0\cr
0&\sin\theta_2
\end{bmatrix}
\end{bmatrix}
\nonumber\\
\sim\ &
\begin{bmatrix}
\cos\theta_1&0\cr
0&\cos\theta_2
\end{bmatrix}+
\begin{bmatrix}
\sin\theta_1&0\cr
0&\sin\theta_2
\end{bmatrix}
b^*H^{-1}a
\begin{bmatrix}
\sin\theta_1&0\cr
0&\sin\theta_2
\end{bmatrix}. \label{equa:lastgeneral}
\end{align}
\endgroup

We observe that
$$b^*H^{-1}a=\left(
\begin{bmatrix}
\cos\theta_1&0\cr
0&\cos\theta_2
\end{bmatrix}+a^*b\right)^{-1}.
$$
We examine the different values of $\cos\theta_{1}$ and $\cos\theta_{2}$.

$\bullet$ First, suppose ``$\cos\theta_1\neq 1$ and $\cos\theta_2\neq 1$".
With usual arguments, we deduce
\begingroup
\addtolength{\jot}{0.5em}
\begin{align*}
P+\diag(0_{k-2},I_{2})
\sim\ &
\begin{bmatrix}
\frac{\cos\theta_1}{\sin^2\theta_1}&0\cr
0&\frac{\cos\theta_2}{\sin^2\theta_2}
\end{bmatrix}+
b^*H^{-1}a
\\
\sim\ &
a^*b\begin{bmatrix}
\cos\theta_1&0\cr
0&\cos\theta_2
\end{bmatrix}
+I_2 \\
\sim\ &
\begin{bmatrix}
\cos\theta_1&0\cr
0&\cos\theta_2
\end{bmatrix}+b^*a.
\end{align*}
\endgroup

Thus the hypothesis on $a^*b$ implies $P\in \Omega(0_{k-2},I_{2})$ in this case.

\medskip

$\bullet$ If $\cos\theta_1=\cos\theta_2=1$, then the hypothesis  implies immediately that 
$P\in \Omega(0_{k-2},I_{2})$.

\medskip

$\bullet$ It only remains to consider $\cos\theta_1=1$ and $\cos \theta_2\neq 1$. (Notice that the  case 
$\cos\theta_1\neq1$ and $\cos \theta_2= 1$ is similar.)
We denote $\theta=\theta_2$. 

--- Suppose $a^*b$ is diagonal, i.e., $a^*b=\begin{bmatrix}
u&0\cr
0&v
\end{bmatrix}
$.
The equality (\ref{equa:lastgeneral}) becomes
\begingroup
\addtolength{\jot}{0.5em}
\begin{align*}
P+\diag(0_{k-2},I_{2})
\sim\ &
\begin{bmatrix}
1&0\cr
0&\cos\theta
\end{bmatrix}
+
\begin{bmatrix}
0&0\cr
0&(v+\cos\theta)^{-1}\sin^2\theta
\end{bmatrix}
\\
\sim\ &
\begin{bmatrix}
1&0\cr
0&\cos\theta+(v+\cos\theta)^{-1}\sin^2\theta
\end{bmatrix}\\
\sim\ &
(v+\cos\theta)^{-1}\left((v+\cos\theta)\cos\theta+\sin^2\theta\right)\\
\sim\ &
1+v\cos\theta .
\end{align*}
\endgroup

 As $\cos\theta\neq 1$, the quaternionic number $1+v\cos\theta $ is  different from 0 and
 $P+\diag(0_{k-2},I_{2})$
 is invertible.
 
 --- If the matrix $a^*b$ is not diagonal, we know from
 \cite[Proposition 5.1]{MR2519914}
 that it has the form
 $a^*b=\begin{bmatrix}
u&-\ov{v}\gamma\cr
v&v\ov{u}v^{-1}\gamma
\end{bmatrix}
$
with $|\gamma|=1$,
$v\neq 0$ and $|v|^2+|u|^2=1$.
The equality (\ref{equa:lastgeneral}) becomes
 \begin{eqnarray*}
P+\diag(0_{k-2},I_{2})
&\sim&
\begin{bmatrix}
1&0\cr
0&\cos\theta
\end{bmatrix}
+
\begin{bmatrix}
0&0\cr
0&\sin\theta
\end{bmatrix}
b^*H^{-1}a
\begin{bmatrix}
0&0\cr
0&\sin\theta
\end{bmatrix}.
\end{eqnarray*}
We compute 
$$b^*H^{-1}a=\left(
\begin{bmatrix}
1&0\cr
0&\cos\theta
\end{bmatrix}
+a^*b\right)^{-1}.$$
Denote 
$$K=
\begin{bmatrix}
1&0\cr
0&\cos\theta
\end{bmatrix}
+a^*b 
=
\begin{bmatrix}
u+1&-\ov{v}\gamma\cr
v&
v\ov{u}v^{-1}\gamma+\cos\theta
\end{bmatrix}
.$$
 If $X$ is such that 
 $$v\ov{u}v^{-1}\gamma+\cos\theta-vX=0,$$
  then we have
\begingroup
\addtolength{\jot}{0.5em}
\begin{align*}
&K^{-1}\\
=&
\begin{bmatrix}
1&-X\cr
0&1
\end{bmatrix}
\,
\begin{bmatrix}
u+1&-\ov{v}\gamma-(u+1)X\cr
v&0
\end{bmatrix}^{-1}\\
=&
\begin{bmatrix}
1&-X\cr
0&1
\end{bmatrix}
\,
\begin{bmatrix}
0&v^{-1}\cr
(-\ov{v}\gamma-(u+1)X)^{-1}&
(\ov{v}\gamma+(u+1)X)^{-1}(u+1)v^{-1}
\end{bmatrix}.
\end{align*}
\endgroup

This implies
\begingroup
\addtolength{\jot}{0.5em}
\begin{align*}
&P+\diag(0_{k-2},I_{2})\\
\sim\ &
\cos\theta(\ov{v}\gamma+(u+1)X)+\sin^2\theta(u+1)v^{-1}\\
\sim\ &
\cos\theta(\ov{v}\gamma+(u+1)\ov{u}v^{-1}\gamma)+\cos\theta(u+1)v^{-1}\cos\theta
+\sin^2\theta(u+1)v^{-1}\\
\sim\ &
\cos\theta(\ov{v}v+u\ov{u}+\ov{u})v^{-1}\gamma+(u+1)v^{-1}\\
\sim\ &
\cos\theta(1+\ov{u})v^{-1}\gamma+(u+1)v^{-1}.
\end{align*}
\endgroup
If this last quaternionic number is equal to zero, we have an equality of modules:
$$(\cos\theta)\,|1+\ov{u}|\,|v|^{-1}=|1+u|\,|v|^{-1}$$
which is impossible since $\cos\theta\neq 1$  and $|1+u|\neq 0$. Therefore, in this last case, we have also
the inversibility of
$ P+\diag(0_{k-2},I_{2})$
and the claim is proven.

\smallskip

\emph{Second step.}
 Now we assume that
$a^*b\notin \Omega
\begin{bmatrix}
\cos\theta_1&0\cr
0&\cos\theta_2
\end{bmatrix}
$.
We observe:\\
$\bullet$  if $a^*b=\begin{bmatrix}
u&0\cr
0&v
\end{bmatrix}$
then the hypothesis implies ($\cos\theta_1=1$ and $u=-1$)
or
($\cos\theta_2=  1$ and $v=-1$).

\medskip

We develop the different cases.

-- Let $a^*b=\begin{bmatrix}
-1&0\cr
0&v
\end{bmatrix}$ with $\cos\theta_1=1$ and $\cos\theta_2\neq 1$.
We denote $\theta=\theta_2$. We replace
$b^*=\begin{bmatrix}
-1&0\cr
0&\ov{v}
\end{bmatrix}a^*$  by its value in the expression of $P$ and get
\begin{equation}
P=
\begin{bmatrix}
m
\begin{bmatrix}
I_{k-4}&0&0\cr
0&1&0\cr
0&0&\cos\theta
\end{bmatrix}
\ell^*
&
m
\begin{bmatrix}
0_{k-4,1}&0_{k-4,1}\cr
0&0\cr
0&-\ov{v}\sin\theta
\end{bmatrix}
a^*
\cr
&\cr
a
\begin{bmatrix}
0_{1,k-4}&0&0\cr
0_{1,k-4}&0&\sin\theta
\end{bmatrix}
\ell^*
&
a
\begin{bmatrix}
-1&0\cr
0&\ov{v}\cos\theta
\end{bmatrix}
a^*\cr
\end{bmatrix}.
\end{equation}
Using the bottom right-hand term as pivot  of $P+\diag(0,-I_2)$ gives, with computations similar to those in the first step,  that
$$P\in \Omega(0_{k-2},-1,-1).$$

-- The second case
with $\cos\theta_2=1$ and $\cos\theta_1\neq 1$, $v=-1$ gives the same result,
$$P\in \Omega(0_{k-2},-1,-1).$$

-- The last case, $\cos\theta_1=\cos\theta_2=1$ corresponds to $P=\begin{bmatrix}
m\ell^*&0\cr
0&ab^*\cr
\end{bmatrix}$
and $P\in \Omega(0_{k-2},-1,-1)$.

\bigskip
$\bullet$ if $a^*b$ is not diagonal, we shall prove that
$\cos\theta_1=\cos\theta_2=1$. In fact, $a^*b$ has the form
$a^*b=\begin{bmatrix}
u&-\ov{v}\gamma\cr
v&v\ov{u}v^{-1}\gamma
\end{bmatrix}
$
with $|\gamma|=1$,
$v\neq 0$ and $|v|^2+|u|^2=1$.
Then, 
 $a^*b\notin \Omega
\begin{bmatrix}
\cos\theta_1&0\cr
0&\cos\theta_2
\end{bmatrix}$
if and only if
$a^*b+\begin{bmatrix}
\cos\theta_1&0\cr
0&\cos\theta_2
\end{bmatrix}=\begin{bmatrix}
u+\cos\theta_1&-\ov{v}\gamma\cr
v&v\ov{u}v^{-1}\gamma+\cos\theta_2
\end{bmatrix}$ is not inversible.

From now on we shall denote $c_1=\cos\theta_1$ and $c_2=\cos\theta_2$.

As $v\neq 0$, we can take the matrix  $X=v^{-1}(v\ov{u}v^{-1}\gamma+c_2)$, so
$$\begin{bmatrix}
u+c_1&-\ov{v}\gamma\cr
v&v\ov{u}v^{-1}\gamma+c_2
\end{bmatrix} \sim
\begin{bmatrix}
u+c_1 &-\ov{v}\gamma-(u+c_1)X\cr
v& 0
\end{bmatrix},$$ 
which is not invertible if and only if 
$$-\ov{v}\gamma-(u+c_1)X=0.$$ 
It follows
\begin{align*}
-\bar v\gamma =& (u+c_1)v^{-1}(v\bar u v^{-1}\gamma + c_2)\\
=&\frac{1}{|v|^2}(|u|^2\bar v\gamma+u\bar v c_2+c_1\bar u \bar v \gamma+\bar v c_1c_2),
\end{align*}
hence
$$-|v|^2\bar v\gamma = |u|^2\bar v\gamma+u\bar v c_2+c_1\bar u \bar v \gamma+\bar v c_1c_2$$
and
$$-\bar v\gamma =u\bar v c_2+c_1\bar u \bar v \gamma+\bar v c_1c_2,$$
because $\vert u\vert^2+\vert v\vert^2=1$.
Finally,
$$ 
-(1+c_1\bar u)\bar v\gamma=(u+c_1)c_2\bar v
$$
and, taking modules,
\begin{equation} \label{Primera_desp}
|1+c_1u|=|u+c_1|c_2.
\end{equation}

We have:
\begin{enumerate}
\item
$c_2\neq 0$: if $c_2=0$, $c_1u=-1$ then $c_1|u|=1$ so $|u|\geq 1$, which is impossible because $v\neq 0$;
\item
let us suppose $c_2<1$: from equation (\ref{Primera_desp}) we have $|1+c_1u|^2 <|u+c_1|^2$ and we deduce $1-c_1^2<(1-c_1^2)|u|^2$, but $1-c_1^2\neq 0$, so $|u|^2>1$, a contradiction;
\item
now, as $c_2=1$, equation (\ref{Primera_desp}) is $|1+c_1u|=|u+c_1|$, which is equivalent to $1-c_1^2=(1-c_1^2)|u|^2$, but $|u|^2<1$, so $c_1^2=1$.
\end{enumerate}
Hence, $\cos\theta_1=\cos\theta_2=1$ as stated.

From Proposition~\ref{prop:sp2}, we deduce the result.
\end{proof}

The previous results lead naturally to the following intriguing question.

\begin{problem}
 Let $k$ and $j$ with $k\geq j$, do we have
 $${\cat}_{X_{2k-j,k}}\Sp(k) \leq \cat\Sp(j)?$$
\end{problem}

Propositions \ref{prop:spnrelative0}, \ref{prop:spnrelative1} and Proposition~\ref{prop:spnrelative2} 
give an affirmative answer for $j=0,\,1,\,2$.

\def\cprime{$'$}
\providecommand{\bysame}{\leavevmode\hbox to3em{\hrulefill}\thinspace}
\providecommand{\MR}{\relax\ifhmode\unskip\space\fi MR }
\providecommand{\MRhref}[2]{%
  \href{http://www.ams.org/mathscinet-getitem?mr=#1}{#2}
}
\providecommand{\href}[2]{#2}

\medskip

\end{document}